\documentclass[11pt,a4paper]{amsart}
\usepackage[T1]{fontenc}
\usepackage[english]{babel}
\usepackage{amsthm,amssymb,amsfonts,mathrsfs,amsmath,fancyhdr}
\usepackage{enumitem}
\usepackage{scalerel}
\usepackage{rotating}
\usepackage{microtype}

\newcommand\blfootnote[1]{%
  \begingroup
  \renewcommand\thefootnote{}\footnote{#1}%
  \addtocounter{footnote}{-1}%
  \endgroup
}

\usepackage[maxbibnames=99,style=numeric, backend=biber, backref, backrefstyle=none, maxalphanames=6]{biblatex}
\newtheorem{maintheorem}{Theorem}

\newenvironment{proof1}{\vspace{0.2cm}\textit{\textbf{Proof  of Theorem \ref{thm:1}.}}}{\hfill \qedsymbol \medskip}

\newenvironment{proof2}{\vspace{0.2cm}\textit{\textbf{Proof  of Theorem \ref{thm:2}.}}}{\hfill \qedsymbol \medskip}

\newenvironment{proof3}{\vspace{0.2cm}\textit{\textbf{Proof  of Theorem \ref{thm:3}.}}}{\hfill \qedsymbol \medskip}

\usepackage[usenames,dvipsnames,usenames,dvipsnames,svgnames,table]{xcolor}
\definecolor{newblue} {RGB} {0,100,200}
\usepackage[nodisplayskipstretch]{setspace}
\usepackage[unicode,pdftex, bookmarks, linktocpage=true]{hyperref}
\usepackage[babel]{csquotes}
\usepackage{microtype}
\usepackage[left=2.7cm, right=2.7cm,top=2.5cm,bottom=2.5cm]{geometry}
\usepackage[all]{xy}
\usepackage{bookmark}
\usepackage{orcidlink}

\usepackage[unicode,pdftex, bookmarks, linktocpage=true]{hyperref}
\hypersetup{hypertexnames=false, colorlinks=true, citecolor=blue, linkcolor=blue, urlcolor=magenta, pdfpagemode=UseNone , breaklinks=true}

\usepackage{tikz}
\usetikzlibrary{matrix,arrows,calc,decorations,patterns,decorations.markings}
\usepackage{tikz-cd}

\usepackage{bm}

\usepackage{etoolbox}
\patchcmd{\section}{\normalfont}{\normalfont\normalsize}{}{}

\usepackage{bbm}
\usepackage{tkz-euclide}
\tikzset{
  dotted/.style={pattern=dots,pattern color=#1},
  dotted/.default=black
}
\tikzset{
  fdotted/.style={pattern=crosshatch dots,pattern color=#1},
  fdotted/.default=black
}
\tikzset{
  scopedlines/.style={pattern=north east lines,pattern color=#1},
  scopedlines/.default=black
}
\tikzset{
  hrlines/.style={pattern=horizontal lines,pattern color=#1},
  hrlines/.default=black
}
\newcommand*{\DashedArrow}[1][]{\mathbin{\tikz [baseline=-0.25ex,-latex, dashed,#1] \draw [#1] (0pt,0.5ex) -- (1.3em,0.5ex);}}
\usepackage{microtype}

\newcommand{\q}[1]{``#1''}

\usepackage{mathtools}

\theoremstyle{plain}
\newtheorem{lem}{Lemma}[section]

\renewcommand\qedsymbol{$\blacksquare$}

\theoremstyle{definition}
\newtheorem{defn}[lem]{Definition}

\theoremstyle{definition}
\newtheorem{rmk}[lem]{Remark}

\theoremstyle{definition}
\newtheorem{con}[lem]{Construction}

\theoremstyle{plain}
\newtheorem{cor}[lem]{Corollary}

\theoremstyle{plain}
\newtheorem{prop}[lem]{Proposition}

\theoremstyle{plain}
\newtheorem{thm}[lem]{Theorem}

\theoremstyle{plain}

\theoremstyle{definition}
\newtheorem{ex}[lem]{Example}

\renewcommand{\mathbb}{\mathbf}

\title[Polygons of Newton-Okounkov type on IHS manifolds]{Polygons of Newton-Okounkov type on irreducible holomorphic symplectic manifolds}
\author{Francesco Antonio Denisi \orcidlink{0000-0002-1128-7890}}
\address{Université Paris Cité and Sorbonne Université, CNRS, IMJ-PRG, F-75013 Paris, France}
\email{denisi@imj-prg.fr}


\begin{document}
\begin{abstract}
Let $X$ be a projective irreducible holomorphic symplectic manifold. We associate with any big $\mathbf{R}$-divisor $D$ on $X$ a convex polygon $\Delta_E^{\mathrm{num}}(D)$ of dimension 2, whose Euclidean volume is $\mathrm{vol}_{\mathbf{R}^2}(\Delta_E^{\mathrm{num}}(D))=q_X(P(D))/2$, where $E$ is any prime divisor on $X$, $q_X$ is the Beauville-Bogomolov-Fujiki form, and $P(D)$ is the positive part of the divisorial Zariski decomposition of $D$. We systematically study these polygons and observe that they behave like the Newton-Okounkov bodies of big divisors on smooth complex projective surfaces, with respect to a general admissible flag.
\end{abstract}

\maketitle
\blfootnote{This project has received funding from the European Research Council (ERC) under the European Union's Horizon 2020 research and innovation programme (ERC-2020-SyG-854361-HyperK).}
 
\paragraph{\textbf{Conventions}} A variety will be an integral, separated scheme of finite type over $\mathbf{C}$.

\section{Introduction}
One of the major steps in the theory of the positivity of divisors on projective varieties is to study the positivity of divisors via convex geometry. More precisely, thanks to the work of Lazarsfeld-Mustață \cite{Mustata} and Kaveh-Kovanskii \cite{KK}, if $Y$ is any projective variety of dimension $d$ and $D$ a big integral Cartier divisor on it (i.e. $h^0(Y,mD) \sim m^d$, for $m\gg 0$), after fixing an admissible flag, i.e.
\[
F_{\bullet}=\{Y_0 \subset Y_1 \subset \cdots \subset Y_d=Y\},
\]
with $Y_i$ a subvariety of $Y$ which is smooth at $Y_0$, and $\mathrm{dim}(Y_i)=i$, for any $0\leq i \leq  d$, one can associate with $D$, via a valuation-like function
\[
\nu=\nu_{F_{\bullet}}=\nu_{F_{\bullet},D} \colon \left(H^0(Y,\mathscr{O}_Y(D))-\{0\}\right)\to \mathbf{Z}^d,
\]
a convex body $\Delta(D)=\Delta_{F_{\bullet}}(D) \subset \mathbf{Z}^d\otimes \mathbf{R}$. The body $\Delta(D)$ is known as the Newton-Okounkov body of $D$ (with respect to the flag $F_{\bullet}$). The first manifestation of the strong relation between $\Delta(D)$ and the positivity of $D$ is given by the equality (see \ \cite[Theorem A]{Mustata}) \[
\mathrm{vol}_{\mathbf{R}^d}(\Delta(D))=\frac{1}{d!}\mathrm{vol}_Y(D). \]
In the equality above, $\mathrm{vol}_{\mathbf{R}^d}(\Delta(D))$ is the euclidean volume of $\Delta(D)$, while $\mathrm{vol}_Y(D)$ is the volume of the divisor $D$, which we recall being defined as
\[
\mathrm{vol}_Y(D)=\limsup_{m \to \infty}\frac{h^0(Y,\mathscr{O}_Y(mD))}{m^d/d!}.
\]
See the work of Küronya-Lozovanu \cite{Kuronya2} for the characterisation of the asymptotic base loci (see Subsection \ref{asympbaseloci}) via Newton-Okounkov bodies. 
\vspace{0.2cm}

Because of the strong relation between the Newton-Okounkov body of a divisor and the positivity of the divisor itself, it is important to study the Newton-Okounkov bodies of divisors. In the case of smooth projective surfaces, the theory is well developed, thanks to the works  \cite{AKL14}, \cite{Mustata}, \cite{LSS14}, \cite{Kuronya1}, \cite{RS22}, \cite{SS16}. 
Given a surface $S$ and a general admissible flag \[
F_{\bullet}=  \{\{pt\}\subset C \subset S\} ,
\] it is known that the Newton-Okounkov body of a big divisor $D$ on $S$ is determined by the way its Zariski decomposition changes when we perturb the class of $D$ in the Néron-Severi space along the negative direction induced by $C$ (see \cite[Theorem 6.4]{Mustata}). More precisely, if $D$ is a big divisor on $S$, the Newton-Okounkov body of $D$ with respect to the flag $F_{\bullet}$ is given by the polygon (see \cite[Theorem 6.4]{Mustata}, \cite[Theorem B]{Kuronya1})
\[
\Delta_{F_{\bullet}}(D)=\left\{(t,y)\in \mathbf{R}^2\; | \; 0\leq t\leq \mu_C(D), \; 0\leq y \leq P(D_t)\cdot C\right\},
\]
where $D_t:=D-tC$ and $\mu_C(D):=\mathrm{sup}\left\{t\in \mathbf{R}^{\geq 0} \; | \; D-tC \text{ is big} \right\}$. Then the Euclidean volume of this body is equal to $P(D)^2/2$, by \cite[Theorem A]{Mustata} and the equality $\mathrm{vol}_S(D)=\mathrm{vol}_S(P(D))$, where $P(D)$ is the positive part of the Zariski decomposition of $D$. 
\vspace{0.2cm}

For a comprehensive treatment of Newton-Okounkov bodies and their applications in the theory of (local) positivity of line bundles on projective varieties, we refer the reader to \cite{Kuronya3}.
\vspace{0.2cm}

We recall that an irreducible holomorphic symplectic (IHS) manifold is a simply connected compact Kähler manifold $X$ such that $H^0(X,\Omega^2_X)\cong \mathbf{C}\sigma$, where $\sigma$ is a holomorphic symplectic form. These objects play a fundamental role in Kähler geometry, and one of the reasons is certainly that, thanks to the Beauville-Bogomolov decomposition Theorem, they form one of the 3 building blocks of compact Kähler manifolds with vanishing first real Chern class. The IHS manifolds in dimension 2 are K3 surfaces. Examples in all dimensions can be constructed starting from a K3 surface or an abelian surface, using some modular constructions. These constructions lead to the deformation classes K3$^{[n]}$, $\mathrm{Kum}_n$, for any $n\geq 1$, and to the OG10, OG6 deformation classes.  
\vspace{0.2cm}

For a general introduction to IHS manifolds, we refer the reader to \cite{Joyce}. Let $X$ be an IHS manifold. Thanks to the work \cite{Beau} of Beauville, there exists a quadratic form on $H^2(X,\mathbf{C})$ generalising the intersection form on a surface. In particular, choosing the symplectic form $\sigma$ of $X$ in such a way that  $\int_X(\sigma\overline{\sigma})^n=1$, one can define
\[
q_X(\alpha):=\frac{n}{2}\int_X (\sigma\overline{\sigma})^{n-1}\alpha^2+(1-n)\left(\int_X\sigma^n\overline{\sigma}^{n-1}\alpha\right)\cdot \left(\int_X\sigma^{n-1}\overline{\sigma}^n\alpha\right),
\]
for any $\alpha \in H^2(X,\mathbf{C})$. The quadratic form $q_X$ is non-degenerate and is known as the Beauville-Bogomolov-Fujiki form (BBF form in what follows). Up to a rescaling, $q_X$ takes integer values and is primitive on $H^2(X,\mathbf{Z})$ (for us $q_X$ will always be integer-valued). Furthermore, there exists a positive constant $c_X \in \mathbf{R}^{>0}$, known as Fujiki constant, such that 
\begin{equation}\label{Fujikirelation}
c_X \cdot q_X(\alpha)^n= \int_X \alpha^{2n},
\end{equation} 
for all $\alpha \in H^2(X,\mathbf{C})$ (see for example \cite[Proposition 23.14]{Joyce}). Formula  (\ref{Fujikirelation}) is known as \textit{Fujiki relation}.
\vspace{0.2cm} 

The BBF form shares many properties with the usual intersection form on a surface. Below are those that will be mostly used:
\vspace{0.2cm}

$\bullet$ $q_X(\alpha)>0$ for any K\"{a}hler class $\alpha$ on $X$ (\cite[1.10]{Huy1}).

$\bullet$ If $\alpha$ is K\"{a}hler, $q_X(\alpha,N)>0$ for any non-zero effective divisor $N$ on $X$ (\cite[1.11]{Huy1}).

$\bullet$ $q_X(E,E')\geq 0$ for any couple of distinct prime divisors on $X$ (\cite[Proposition 4.2, item (ii)]{Bouck}).
\vspace{0.2cm}

Note that $q_X$ is invariant under deformation and its signature on $H^2(X,\mathbf{R})$ is $\left(3,b_2(X)-3\right)$, while, when $X$ is projective, its signature on $N^1(X)_{\mathbf{R}}$ is $(1,\rho(X)-1)$, where $\rho(X)$ is the Picard number of $X$.
\vspace{0.2cm}

In the case of projective IHS manifolds, there exists a decomposition of the big cone into locally rational polyhedral subcones (see Definition \ref{locratpol}), called Boucksom-Zariski chambers (see \cite[Theorem 1.6]{Den}).\ This decomposition is analogous to that for the big cone of smooth projective surfaces, provided by Bauer-Küronya-Szemberg in \cite{Bau}, and explains how the divisorial Zariski decomposition of a big divisor changes in the big cone when the chosen divisor class is perturbed. For the definition of the divisorial Zariski decomposition of any pseudo-effective divisor on a projective IHS manifold, we refer the reader to Subsection \ref{divZardecomp}. Using the mentioned decomposition of the big cone and the BBF pairing, we can associate with any big divisor a polygon behaving like the Newton-Okounkov body of a big divisor on a smooth complex projective surface (computed with respect to a general admissible flag). In particular, let $X$ be a projective IHS manifold, $D$ a big divisor on $X$, $E$ a prime divisor on $X$ and $D=P(D)+N(D)$ the divisorial Zariski decomposition of $D$. Consider the subset of $\mathbf{R}^2$
\[
\Delta_{E}^{\mathrm{num}}(D):=\left\{(t,y)\in \mathbf{R}^2\; | \; 0\leq t\leq \mu_E(D), \; 0\leq y \leq q_X(P(D_t),E)\right\},
\]
where $D_t:=D-tE$ and $\mu_E(D):=\mathrm{sup}\{t\in \mathbf{R}^{\geq 0} \;|\; D-tE \text{ is pseudo-effective}\}$.

\begin{maintheorem}\label{thm:1}
The subset $\Delta_E^{\mathrm{num}}(D)$ of $\mathbf{R}^2$ introduced above is a convex polygon. Moreover, $\mathrm{vol}_{\mathbf{R}^2}(\Delta_E^{\mathrm{num}}(D))=q_X(P(D))/2$, where $\mathrm{vol}_{\mathbf{R}^2}(\Delta_E^{\mathrm{num}}(D))$ is the Euclidean volume of the polygon $\Delta_E^{\mathrm{num}}(D)$. In particular, we have
\[
2^nc_X(\mathrm{vol}_{\mathbf{R}^2}(\Delta^{\mathrm{num}}_E(D))^n=c_X(q_X(P(D)))^n=\mathrm{vol}_X(D)=(2n)!\cdot \mathrm{vol}_{\mathbf{R}^{2n}}(\Delta(D)),
\] 
where $\Delta(D)$ is the classical Newton Okounkov body of $D$, with respect to any admissible flag.
\end{maintheorem}

As in the case of the classical Newton-Okounkov bodies (see \cite[Theorem B]{Mustata}), we show the existence of a convex cone of dimension $\rho(X)+2$ (which means that it belongs to $\mathbf{R}^{\rho(X)+2}$ and its interior is non-empty), whose "slices" are the polygons $\Delta_E^{\mathrm{num}}(D)$ in the statement above.

\begin{maintheorem}\label{thm:2}
For any prime divisor $E$ on a projective IHS manifold $X$ there exists a convex cone $\Delta^{\mathrm{num}}_E(X) \subset N^1(X)_{\mathbf{R}} \times \mathbf{R}^2$, together with the natural projection $ p_1 \colon \Delta^{\mathrm{num}}_E(X) \to N^1(X)_{\mathbf{R}}$, such that the fibre $p_1^{-1}(\zeta)$ is equal to the polygon $\Delta^{\mathrm{num}}_E(\zeta)$, for all $\zeta$ in $\mathrm{Big}(X)$.
\end{maintheorem}

Following \cite{LSS14}, we show that under some assumptions, the polygons $\Delta_E^{\mathrm{num}}(\alpha)$ (with $\alpha$ a big rational class) can be expressed as a Minkowski sum of some polygons $\{\Delta_E^{\mathrm{num}}(\beta_i)\}_{i \in I}$, for some movable rational classes $\{\beta_i\}_{_ i\in I}$. More precisely, we have the following.

\begin{maintheorem}\label{thm:3}
Let $X$ be a projective IHS manifold with $\mathrm{Eff}(X)$ rational polyhedral and $E$ a big prime divisor. There exists a finite set $\Omega$ of movable $\mathbf{Q}$-divisors such that for any big $\mathbf{Q}$-divisor $D$ there exist rational numbers $\{\alpha_P(D)\}_{P\in \Omega}$ such that, $P(D)=\sum_{P \in \Omega}\alpha_P(D)P$ and, up to a translation, $\Delta_E^{\mathrm{num}}(D)=\sum_{P\in \Omega}\alpha_P(D)\Delta_E(P)$.
\end{maintheorem}

\begin{rmk}
We notice that the pseudo-effective cone of a projective IHS manifold  $X$ is rational polyhedral if and only if the set of prime exceptional divisors of $X$ is finite and non-empty when $\rho(X)\geq 3$ (see \cite[Theorem 1.2]{Den2}). If $\rho(X)=2$, $\overline{\mathrm{Eff}(X)}$ is rational polyhedral if and only if $\mathrm{Bir}(X)$ is finite (see \cite[Theorem 1.3, item (2)]{Ogu} and \cite[Corollary 1.6]{Den2}).
\end{rmk}

Newton-Okounkov bodies have attracted a great deal of interest in recent years. These objects contain a lot of asymptotic information about the divisors on projective varieties, but in general, they are not easy to compute. In addition, the dimension of classical Newton-Okounkov bodies is equal to the dimension of the variety. In the case of projective IHS manifolds, we have the classical Newton-Okounkov bodies and the polygons of Newton-Okounkov type in Theorem \ref{thm:1} above. These polygons are easier to compute (we present several such computations in Section \ref{examples}). We hope they can be used to obtain finer information on the positivity of divisors on projective IHS manifolds.

 \section*{Acknowledgements}
Firstly, I would like to thank Gianluca Pacienza and Giovanni Mongardi for all their support.\ Secondly, I would like to thank Sébastien Boucksom and Alex Küronya for some interesting discussions. Finally, I would like to thank the anonymous referee from the \q{Annali della Scuola Normale Superiore di Pisa - Classe di Scienze} for helping me to improve the exposition of the article.

\section*{Organisation of the paper}
\begin{enumerate}
\item In the second section we collect all the notions and known results needed.
\item In the third section we compute the restricted volume of big divisors on projective IHS manifolds along prime divisors.
\item In the fourth section, we associate with any big divisor on a projective IHS manifold a convex polygon. We say that these polygons are of \textit{Newton-Okounkov type}. Furthermore, we observe that these polygons fit in a convex cone as slices. To conclude, we observe that in some cases we can explicitly compute a set of generators for this convex cone.
\item In the fifth section, we investigate when the polygons defined in the fourth section can be expressed as Minkowski sums of simpler ones.
\item In the sixth section we give several examples illustrating our results.
\end{enumerate}
\vspace{0.6cm}

\section{Preliminaries}
Throughout this section, $Y$ will be a normal projective variety of dimension $d$. A prime divisor on $Y$ is a reduced and irreducible hypersurface. From now on, $X$ will denote a projective IHS manifold of complex dimension $2n$.
\vspace{0.2cm}

The following formula will be useful.

\begin{equation}\label{relation}
q_X(\alpha,\beta)\int\alpha^{2n}=q_X(\alpha)\int\alpha^{2n-1}\beta, \text{ for any } \alpha,\beta \in H^2(X,\mathbb{C})
\end{equation}
For the above formula, see for example \cite[Exercise 23.2]{Joyce}, but be careful, there is a typo in there.

\begin{defn}
We say that a prime divisor $E$ on $X$ is exceptional if $q_X(E)<0$.
 \end{defn}

  Recall that the N\'{e}ron-Severi space  of $Y$ can be defined as $N^1(Y)_{\mathbf{R}}:=N^1(Y) \otimes_{\mathbf{Z}} \mathbf{R}$, where $N^1(Y):=\mathrm{Div}(Y)/\equiv$ is the  N\'{e}ron-Severi group of $Y$, $\mathrm{Div}(Y)$ is the group of integral Cartier divisors of $Y$, and $\equiv$ is the numerical equivalence relation. A Cartier $\mathbf{R}$-divisor is a formal linear combination (with coefficients in $\mathbf{R}$) of integral Cartier divisors. The class of a Cartier divisor $D$ in $N^1(Y)_{\mathbf{R}}$ will be denoted by $[D]$.

 \begin{rmk}
 For an IHS manifold $X$, since $H^1(X,\mathscr{O}_X)=0$, the numerical and linear equivalence relations coincide.
 \end{rmk}
  
 \begin{defn}{}{}
 A Cartier $\mathbf{Q}$-divisor is big if there exists a positive integer $m$ such that $mD$ is a big integral Cartier divisor. If $D$ is any Cartier $\mathbf{R}$-divisor, we say that it is big if $D=\sum_ia_iD_i$, where $a_i$ is a positive real number and $D_i$ is a big integral Cartier divisor, for any $i$.
 \end{defn}
 
  An ample Cartier $\mathbf{R}$-divisor is a Cartier $\mathbf{R}$-divisor $D$ of the form $D=\sum_ia_iD_i$, where any $a_i$ is a positive real number, and any $D_i$ is an ample integral Cartier divisor. A Cartier $\mathbf{R}$-divisor $D$ on $Y$ is nef if $D \cdot C \geq 0$, for any irreducible and reduced curve $C$ on $Y$. A Cartier $\mathbf{R}$-divisor $D$ is effective if $D=\sum_ia_iD_i$, $a_i$ is a non-negative real number, and $D_i$ is a prime Cartier divisor, for any $i$. 
 \vspace{0.2cm}
 
 The ample cone $\mathrm{Amp}(Y)$ is the convex cone in $N^1(Y)_{\mathbf{R}}$ of ample Cartier $\mathbf{R}$-divisor classes. The nef cone $\mathrm{Nef}(Y)$ is the convex cone in $N^1(Y)_{\mathbf{R}}$ of nef Cartier $\mathbf{R}$-divisor classes. It is known that $\mathrm{Nef}(Y)=\overline{\mathrm{Amp}(Y)}$, and that $\mathrm{Amp}(Y)$ is the interior of  $\mathrm{Nef}(Y)$ (see Theorem 1.4.23 of \cite{Laz}).
 The big cone $\mathrm{Big}(Y)$ of $Y$ is the convex cone in $N^1(Y)_{\mathbf{R}}$ of big Cartier $\mathbf{R}$-divisor classes. The pseudo-effective cone of $Y$ is the closure of the big cone in $N^1(Y)_{\mathbf{R}}$. Moreover, the big cone is the interior of the pseudo-effective cone. The effective cone $\mathrm{Eff}(Y)$ of $Y$ is the convex cone in $N^1(Y)_{\mathbf{R}}$ spanned by effective Cartier $\mathbf{R}$-divisor classes. One can equivalently define the pseudo-effective cone of $Y$ as the closure $\overline{\mathrm{Eff}(Y)}$ of the effective cone in the Néron-Severi space. We have $\mathrm{Amp}(Y)\subset \mathrm{Big}(Y)$, and so $\mathrm{Nef}(Y)\subset \overline{\mathrm{Eff}(Y)}$.
  
  \begin{defn}
\begin{enumerate}
\item  A movable integral Cartier divisor on $Y$ is an effective integral Cartier divisor $D$ such that the linear series $|mD|$ has no divisorial components in its base locus, for any large enough and sufficiently divisible $m$.
\item We define the movable cone $\mathrm{Mov}(Y)$ in $N^1(Y)_{\mathbf{R}}$ as the convex cone generated by the movable integral divisor classes. 
\end{enumerate} 
\end{defn}

\begin{defn}
An $\mathbf{R}$-divisor $D$ on $X$ is $q_X$-nef if $q_X(D,E)\geq 0$ for any prime divisor $E$ on $X$.
\end{defn}

On projective IHS manifolds, $q_X$-nef divisor classes lie in $\overline{\mathrm{Mov}(X)}$. Vice versa, any class $\alpha$ in $\overline{\mathrm{Mov}(X)}$ is $q_X$-nef (see \cite[Lemma 2.7, Remark 2.10]{Den}). 
\vspace{0.2cm}

We now recall a definition concerned with the geometry of convex cones.

\begin{defn}\label{locratpol}
Let $K\subset \mathbf{R}^k$ be a closed convex cone with a non-empty interior. We say that $K$ is (rational) polyhedral if $K$ is generated by finitely many (rational) elements of $\mathbf{R}^k$. We say that $K$ is locally (rational) polyhedral at $v \in \partial K$ if $v$ has an open neighbourhood $U = U(v)$, such that $K \cap U$ is defined in $U$ by a finite number of (rational) linear inequalities. If $K$ is not closed, we say that it is locally (rational) polyhedral at $v \in \partial K \cap K$ if $\overline{K}$ is.
\end{defn}
  
\subsection{Asymptotic base loci}\label{asympbaseloci}
We recall the definition of the various asymptotic base loci and refer the reader to \cite{Ein} for a comprehensive treatment of the argument. The stable base locus of an integral Cartier divisor $D$ is defined as
\[
\mathbf{B}(D)=\bigcap_{k\geq 1}\mathrm{Bs}(|kD|),
\] 
where $\mathrm{Bs}(|kD|)$ is the base locus of the linear series $|kD|$. The stable base locus of a Cartier $\mathbf{Q}$-divisor $D$ on $Y$ is defined as $\mathbf{B}(mD)$, where $m$ is any positive integer such that $mD$ is integral.

\begin{defn}{}{}
Let $D$ be any Cartier $\mathbf{R}$-divisor on $Y$.
\begin{itemize}
\item 
The real stable base locus of $D$ is
\[
\mathbf{B}(D):=\bigcap\{\mathrm{Supp}(E)\;|\; E \text{ effective $\mathbf{R}$-divisor, $E \sim_{\mathbf{R}} D$ }\}.
\]
If $D$ is a $\mathbf{Q}$-divisor, the real stable base locus of $D$ coincides with the usual one (\cite[Proposition 1.1]{BBP2013}).

\item The augmented base locus of $D$ is
\[
\mathbf{B}_{+}(D):=\bigcap_{D=A+E}\mathrm{Supp}(E),
\]
where the intersection is taken over all the decompositions of the form $D=A+E$, where $A$ and $E$ are Cartier $\mathbf{R}$-divisors such that $A$ is ample and $E$ is effective.

\item The restricted base locus of $D$ on $Y$ is
\[
\mathbf{B}_{-}(D):=\bigcup_{A}\mathbf{B}(D+A),
\]
where the union is taken over all ample $\mathbf{R}$-divisors $A$. 
\end{itemize}
\end{defn}

If $D$ is any Cartier $\mathbf{R}$-divisor on $Y$, we have the inclusions $\mathbf{B}_-(D) \subset \mathbf{B}(D)\subset \mathbf{B}_+(D)$ (see \cite[Example 1.16]{Ein}, and \cite[Introduction]{BBP2013}).
\vspace{0.2cm}

\begin{rmk}
The augmented and restricted base loci on $Y$ can be studied in the Néron-Severi space because they do not depend on the representative of the class we choose.
\end{rmk}

  \subsection{Volumes}
 The volume of an integral big Cartier divisor $D$ on $Y$, which we defined in the introduction, measures the asymptotic rate of growth of the spaces of global sections $H^0(Y,\mathscr{O}_Y(mD))$. It is known that $\mathrm{vol}(D)>0$ if and only if $D$ is big. One defines the volume of a Cartier $\mathbf{Q}$-divisor $D$ by picking an integer $k$ such that $kD$ is integral, and by setting $\mathrm{vol}(D):=\frac{1}{k^{d}}\mathrm{vol}(kD)$. This definition does not depend on the integer $k$ we choose (see Lemma 2.2.38 on \cite{Laz}). One can extend the notion of volume to every Cartier $\mathbf{R}$-divisor. In particular, pick $D \in \mathrm{Div}_{\mathbf{R}}(Y)=\mathrm{Div}(Y)\otimes_{\mathbf{Z}} \mathbf{R}$ and let $\{D_k\}_k$ be a sequence of Cartier $\mathbf{Q}$-divisors converging to $D$ in $N^1(Y)_{\mathbf{R}}$. We define
\[
\mathrm{vol}(D):=\lim_{k \to \infty}\mathrm{vol}(D_k).
\]
This number is independent of the choice of the sequence $\{D_k\}_k$ (see \cite[Theorem 2.2.44]{Laz}). Moreover, two numerically equivalent Cartier divisors have the same volume (see \cite[Proposition 2.2.41]{Laz}), hence the volume can be studied in $N^1(Y)_{\mathbf{R}}$, and by \cite[Theorem 2.2.44 ]{Laz} it defines a continuous function $\mathrm{vol}_Y(-) \colon N^1(Y)_{\mathbf{R}} \to \mathbf{R}$, called the \textit{volume function}.
\vspace{0.2cm}

Now, let $V$ be a subvariety of $Y$ of dimension $m>0$, $D$ an integral Cartier divisor on $Y$. We have the natural restriction maps
\[
r_k\colon H^0(Y,\mathcal{O}_Y(kD))\to H^0(V,\mathcal{O}_V(kD)), 
\]
for any positive integer $k$. We define the restricted volume of $D$ to $V$ as 
\[
\mathrm{vol}_{X|V}(D):=\mathrm{lim sup}_{k \to \infty}\frac{\mathrm{rank}(r_k)}{k^m}.
\]
If $D$ is any Cartier $\mathbf{Q}$-divisor, one picks $k$ such that $kD$ is integral, and defines $\mathrm{vol}_{X|V}(D):=\frac{1}{k^m}\mathrm{vol}_{X|V}(kD)$. To conclude, the following result generalises the listed properties of the classical volumes to the case of restricted volumes.
\begin{thm}[Theorem A, \cite{Ein1}]
Let $V \subset X$ be a subvariety of dimension $m>0$ and let $D$ be a Cartier $\mathbf{Q}$-divisor such that $V \not \subset \mathbf{B}_+(D)$. Then $\mathrm{vol}_{X|V}(D)>0$ and $\mathrm{vol}_{X|V}(D)$ depends only on the numerical equivalence class of $D$. Furthermore, $\mathrm{vol}_{X|V}(D)$ varies continuously as a function of the numerical equivalence class of $D$, and it extends uniquely to a continuous function
\[
\mathrm{vol}_{X|V} \colon \mathrm{Big}^V(X)^+\to \mathbf{R}^{\geq 0},
\]
where $\mathrm{Big}^V(X)^+$ denotes the set of all Cartier divisor classes $\zeta$ such that $V \not \subset \mathbf{B}_+(\zeta)$.
\end{thm}

\subsection{Divisorial Zariski decompositions}\label{divZardecomp}
On any compact complex manifold we have the divisorial Zariski decomposition, which was introduced by Boucksom in his fundamental paper \cite{Bouck} (but see also Nakayama's construction, \cite{Nakayama}), and is a generalisation of the classical Zariski decomposition on surfaces. The divisorial Zariski decomposition plays a central role in this paper, and in the case of IHS manifolds Boucksom characterised it with respect to the BBF form. A divisorial Zariski decomposition exists for arbitrary pseudo-effective analytic classes on compact complex manifolds, but as in this paper we deal with projective IHS manifolds, we recall what is a (actually, the) divisorial Zariski decomposition only for pseudo-effective divisors on projective IHS manifolds.

\begin{thm}[Theorem 4.8, \cite{Bouck}]\label{DivZarDec}
Let $D$ be a pseudo-effective $\mathbf{R}$-divisor on $X$ projective. Then $D$ admits a divisorial Zariski decomposition, i.e.\ we can write $D=P(D)+N(D)$ in a unique way such that:
\begin{enumerate}
\item $P(D) \in \overline{\mathrm{Mov}(X)}$,
\item $N(D)$ is an effective $\mathbf{R}$-divisor. If non-zero, $N(D)$ is exceptional, which means that the Gram matrix (with respect to $q_X$) of its irreducible components is negative definite, 
\item $P(D)$ is orthogonal (with respect to $q_X$) to any irreducible component of $N(D)$.
\end{enumerate}
The divisor $P(D)$ (resp.\ $N(D)$) is the \textit{positive part} (resp.\ \textit{negative part}) of $D$.
\end{thm}

\begin{rmk}\label{Rem:zardecomp}
 If the divisor $D$ is effective, its positive part can be defined as the maximal $q_X$-nef subdivisor of $D$. This means that, if $D=\sum_ia_iD_i$, where the $D_i's$ are the irreducible components of $D$, the positive part of $D$ is the maximal divisor $P(D)=\sum_ib_iD_i$, with respect to the relation "$P(D)\leq D$ if and only if $b_i \leq a_i$ for every $i$", such that $q_X(P(D),E)\geq 0$, for any prime divisor $E$. For a proof of this fact see for example the discussion after \cite[Corollary 4.3]{BCK09}, and the proof of \cite[Theorem 3.6]{KMPP19}.
\end{rmk}

\begin{rmk} The divisorial Zariski decomposition is rational, namely, under the assumptions of Theorem \ref{DivZarDec}, whenever $D$ is a pseudo-effective $\mathbf{Q}$-divisor, $P(D)$ and $N(D)$ are $\mathbf{Q}$-divisors.
\end{rmk}

\begin{rmk}\label{rmk3} We note that the divisorial Zariski decomposition behaves well with respect to the numerical equivalence relation. Indeed, let $X$ be a projective IHS manifold and $\alpha$ a pseudo-effective class in $N^1(X)_{\mathbf{R}}$. Suppose $\alpha = [D_1] = [D_2]$, where $D_1,D_2$ are pseudo-effective $\mathbf{R}$-divisors. Write $D_i=P(D_i)+N(D_i)$ for the divisorial Zariski decomposition of $D_i$. Then $D_1-D_2=T$, where $T$ is numerically trivial. Thus 
\[
P(D_1)+N(D_1)=(P(D_2)+T)+N(D_2),
\] 
and by the uniqueness of the divisorial Zariski decomposition, we have $P(D_1)=P(D_2)+T$, because $[P(D_2)+T]$ is $q_X$-nef, $q_X$-orthogonal to $N(D_1)$ and $N(D_1)$ is exceptional. In particular, $[P(D_1)]=[P(D_2)]$ and $N(D_1)=N(D_2)$.
\end{rmk}
 
 The positive part of the divisorial Zariski decomposition of a big divisor encodes most of the positivity of the divisor itself. 
Indeed, given any big $\mathbf{Q}$-divisor $D$ on $X$, one has $H^0(X,mD)\cong H^0(X,mP(D)) $, for $m$ positive and sufficiently divisible (\cite[Theorem 5.5]{Bouck}), and also $\mathbf{B}_+(D)=\mathbf{B}_+(P(D))$, $\mathbf{B}_-(D)=\mathbf{B}_-(P(D))\cup \mathrm{Supp}(N(D))$  (see \cite[Proposition 2.1, Proposition 2.4]{DO23}). Furthermore, $D$ is big if and only if $P(D)$ is (see \cite[Proposition 3.8]{Bouck} or \cite[Proposition 3.8]{KMPP19}).

\section{Restricting volumes along prime divisors}
In this section, we use the formula (\ref{relation}) to compute the restricted volume of a big $\mathbf{R}$-divisor on $X$ along a prime divisor $E$. We first recall the following definition.

\begin{defn}
A small $\mathbf{Q}$-factorial modification of a normal projective variety $Y$ is a birational map $g\colon Y \DashedArrow Y'$ where $Y'$ is normal, projective, and $\mathbf{Q}$-factorial, and $g$ is an isomorphism in codimension 1.
\end{defn}

\begin{lem}\label{lem1}
Let $E$ be a prime divisor on a normal projective variety $Y$, and $f\colon Y \DashedArrow Y'$ a small $\mathbf{Q}$-factorial modification of $Y$. Suppose that, given a Cartier divisor $D$ on $Y$, the Weil divisor $D'=f_{*}(D)$ is Cartier. Furthermore, let $E'=f_{*}(E)$ be the strict tranform of $E$ via $f$. Then the complex vector spaces \[\mathrm{im}(H^0(Y,\mathcal{O}_Y(mD))\to H^0(E,\mathcal{O}_E(mD))),\; \mathrm{im}(H^0(Y',\mathcal{O}_{Y'}(mD'))\to H^0(E',\mathcal{O}_{E'}(mD')))\] have the same dimension, for any $m$, so that $\mathrm{vol}_{X|E}(D)=\mathrm{vol}_{X'|E'}(D')$.
\end{lem}
\begin{proof}
A section $s \in H^0(Y,\mathcal{O}_Y(mD))$ vanishes on $E$ if and only if $s'=f_{*}(s)\in H^0(Y',\mathcal{O}_{Y'}(mD'))$ vanishes on $E'$. As $H^0(Y,\mathcal{O}_Y(mD)) \cong H^0(Y',\mathcal{O}_{Y'}(mD'))$ (because $f$ is a small $\mathbf{Q}$-factorial modification) we are done.
\end{proof}

\begin{prop}\label{prop:restvol}
Let $E$ be a prime divisor on a projective IHS manifold $X$ and $D$ any big divisor with $E\not\subset \mathbf{B}_+(D)$. Then \[
\mathrm{vol}_{X|E}(D)=P(D)^{2n-1}\cdot E=\frac{q_X(P(D),E)\cdot \mathrm{vol}(D)}{q_X(P(D))}.\]
\end{prop}
\begin{proof}
Let $f\colon X\DashedArrow X'$ be a birational map to a projective IHS manifold $X'$ such that $f_{*}(P(D))=P(D')$ is nef ($D'=f_*(D)$), and set $E'=f_{*}(E)$ (this birational map exists by \cite[Theorem 1.2]{MZ13}). As $\mathrm{Big}^{E}(X)^+\subset \mathrm{Big}(X)$ is open, we can find a sequence $\{\gamma_n\}_n$ of big $\mathbf{Q}$-divisor classes converging to $[D]=\gamma$, which in turn gives (via push-forward) a sequence of big $\mathbf{Q}$-divisor classes $\{\gamma'_k\}_k$ in $\mathrm{Big}^{E'}(X')^+$, converging to $[D']=\gamma'$. 
By \cite[Proposition 3.1]{Matsumura} we have $\mathrm{vol}_{X|E}(\gamma_k)=\mathrm{vol}_{X|E}(P(\gamma_k))$. Moreover, by Lemma \ref{lem1}, we have $\mathrm{vol}_{X|E}(\gamma_k)=\mathrm{vol}_{X'|E'}(\gamma'_k)$, and hence $\mathrm{vol}_{X|E}(P(\gamma))=\mathrm{vol}_{X'|E'}(P(\gamma'))$.
By the continuity of the divisorial Zariski decomposition (see \cite[Proposition 4.31]{Den}) we obtain $\{P(\gamma_k)\}_k \to P(\gamma)$, and $\{P(\gamma'_k)\}_k \to P(\gamma')$. By the continuity of $\mathrm{vol}_{X|E}(-)$ (resp.\ $\mathrm{vol}_{X'|E'}(-)$) in $\mathrm{Big}^{E}(X)^+$ (resp.\ $\mathrm{Big}^{E'}(X')^+$)  we conclude that $\mathrm{vol}_{X|E}(D)=\mathrm{vol}_{X'|E'}(D')$. Now, by \cite[Corollary 3.3]{Matsumura}, we have \[\mathrm{vol}_{X'|E'}(D')=\mathrm{vol}_{X'|E'}(P(D'))=P(D')^{2n-1}\cdot E',\] and using formula (\ref{relation}) we have 
\[
q_{X'}(P(D'),E')\cdot (P(D'))^{2n}=q_{X'}(P(D'))\cdot (P(D')^{2n-1}\cdot E').
\]
In particular, using the Fujiki relation and the equality $\mathrm{vol}_{X'|E'}(P(D'))=P(D')^{2n-1}\cdot E'$, we obtain
\[
c_{X'}\cdot (q_{X'}(P(D')))^n\cdot q_{X'}(P(D'),E')=q_{X'}(P(D'))\cdot \mathrm{vol}_{X'|E'}(P(D')).
\]
On $X$ we have 
\[
c_{X}\cdot (q_X(P(D)))^n \cdot q_{X}(P(D),E)=q_{X}(P(D))\cdot (P(D)^{2n-1}\cdot E),
\]
and it is known that $f$ induces an isometry between $\mathrm{Pic}(X)$ and $\mathrm{Pic}(X')$ with respect to the BBF forms (see for example \cite[Proposition I.6.2]{OGrady}), so that $q_{X'}(P(D'),E')=q_{X}(P(D),E)$, and $q_{X}(P(D))=q_{X}(P(D'))$. Moreover, the Fujiki constant is a birational invariant, thus $c_X=c_{X'}$. Putting everything together we obtain 
\begin{align}\label{eqn1}
\mathrm{vol}_{X'|E'}(P(D'))&=c_{X'}(q_{X'}(P(D')))^{n-1}q_{X'}(P(D'),E')\\&=c_{X}(q_{X}(P(D)))^{n-1}q_{X}(P(D),E)=P(D)^{2n-1}\cdot E  \nonumber.
\end{align}
But $\mathrm{vol}_{X'|E'}(P(D'))=\mathrm{vol}_{X|E}(P(D))$, hence $\mathrm{vol}_{X|E}(P(D))=P(D)^{2n-1}\cdot E$.
We conclude that $$\mathrm{vol}_{X|E}(D)=P(D)^{2n-1}\cdot E=\frac{q_X(P(D),E)\cdot \mathrm{vol}(D)}{q_X(P(D))}$$ by using the equality $\mathrm{vol}(D)=P(D)^{2n}=c_{X}\cdot (q_X(P(D)))^n$ (see \cite[Proposition 4.12]{Bouck}).
\end{proof}

Let $Y$ be a normal complex projective variety. It was proven in \cite{Den} that the volume function on a projective IHS manifold is piecewise polynomial. Using that for a prime divisor $E$ the function $\mathrm{vol}_{Y|E}(-)$ (the restricted volume function along $E$) is the derivative of $\mathrm{vol}_Y(-)$ along the direction in $N^1(Y)_{\mathbf{R}}$ induced by $E$ (see \cite[Corollary C]{Mustata} or \cite[Corollary C]{BFJ}), one deduces that if the volume function is piecewise, also $\mathrm{vol}_{Y|E}(-)$ is. As a consequence of the above proposition, we provide below a proof of the local polynomiality of $\mathrm{vol}_{X|E}(-)$, without using the mentioned differentiability result. 

\begin{cor}
The function $\mathrm{vol}_{X|E}(-)\colon \mathrm{Big}^{E}(X)^+ \to \mathbf{R}^{\geq 0} $ is piecewise. Furthermore, $\mathrm{vol}_{X|E}(-)$ takes rational values on big, $\mathbf{Q}$-divisor classes.
\end{cor}
\begin{proof}
Choose a Boucksom-Zariski chamber $\Sigma_S$ (see \cite[Definition 4.11]{Den}), associated with an exceptional set of prime divisors $S=\{E_1,\dots,E_k\}$ not containing $E$, and let $\mathcal{B}=\{P_1,\dots,P_j\}$ be a basis for $S^{\perp}$. Here "exceptional set" means that the Gram matrix of $S$ is negative definite (in particular, any element of $S$ is a prime exceptional divisor). Then, using equation (\ref{eqn1}), we obtain
\[
\left(\mathrm{vol}_{X|E}(-)\right)_{|\Sigma_S}=c_{X}\left(q_{X}\left(\sum_{i=1}^jx_iP_i\right)\right)^{n-1}\cdot q_{X}\left(\sum_{i=1}^jx_iP_i,E\right),
\]
which is a homogeneous polynomial of degree $2n-1$. Then $\mathrm{vol}_{X|E}(-)$ is piecewise, because $\mathrm{Big}^E(X)^+$ is covered by the Boucksom-Zariski chambers of the type $\Sigma_S$, with $S$ such that $E \not\in S$.
The rationality of $\mathrm{vol}_{X|E}(-)$ on $\mathbf{Q}$-divisor classes follows by the rationality of the Fujiki constant and $q_X$ on $N^1(X)_{\mathbb{Q}}$.
\end{proof}

\section{Polygons of Newton-Okounkov-type}
In this section, we associate with any big divisor $D$ on a projective IHS manifold $X$ a $2$-dimensional convex body. We compute the Euclidean volume of these bodies, and we deduce that they are in fact (possibly irrational) polygons.

\begin{prop}\label{prop1}
Let $D$ be any big $\mathbf{R}$-divisor on $X$. Then a prime divisor $E$ is an irreducible component of $\mathbb{B}_{+}(D)=\mathbb{B}_{+}(P(D))$ if and only if $E$ is exceptional and
\[
E \in \mathrm{Null}_{q_X}(P(D)):=\{E' \text{ prime divisor }| \; q_X(P(D),E')=0\}.
\]
\end{prop}
\begin{proof}
See \cite[Remark 2.3, Proposition 2.4]{DO23}.
\end{proof}

\begin{defn}
Let $D$ be any pseudo-effective $\mathbf{R}$-divisor on $X$, and $E$ a prime divisor. We define the threshold
\[
\mu_E(D):=\mathrm{sup}\{t\in \mathbf{R}^{\geq 0} \;|\; D_t \text{ is pseudo-effective}\},\]
where $D_t:=D-tE$. The definition makes sense also for pseudo-effective classes, and we will adopt the same notation in this case.
\end{defn}

\begin{lem}\label{rmk1}
Let $E$ be a prime divisor on $X$, and $D$ a big $\mathbf{R}$-divisor. Then $E \not\subset \mathbb{B}_{-}(D)$ implies $E \not\subset \mathbb{B}_{+}(D_t)$, for any $t\in ]0,\mu_E(D)[$. 
\begin{proof}
Recall that the divisorial part of $\mathbb{B}_{-}(D)$ equals the union of the irreducible components of $N(D)$ (see for example \cite[Theorem 4.1]{Kuronya2}). If $q_X(E)\geq 0$, there is nothing to prove, since the divisorial augmented base locus of big divisors on $X$ is made of prime exceptional divisors. By Proposition \ref{prop1}, it suffices to prove that $E$ does not belong to $\mathrm{Null}_{q_X}(P(D_t))$. Suppose by contradiction it does. We have $P(D)+N(D)=P(D_t)+N(D_t)+tE$, and we can find a big and $q_X$-nef integral divisor $M$ satisfying $\mathrm{Null}_{q_X}(P(D_t))=\mathrm{Null}_{q_X}(M)$ (see \cite[Lemma 4.15]{Den}). Then the Hodge-index Theorem implies that the Gram matrix of $N(D_t)+tE$ is negative definite, which implies $N(D)=N(D_t)+tE$, and this is absurd, because $E \not \subset \mathbb{B}_{-}(D)$ by assumption.
\end{proof}
\end{lem}

\begin{con}\label{construction1}
Let $D$ be any big $\mathbf{R}$-divisor on $X$ and $E$ a prime divisor not lying in $\mathbb{B}_{-}(D)$. Then $q_X(P(D),E)>0$ for any $t \in ]0,\mu_E(D)[$, by the remark above. We define (inspired by \cite[Exercise 2.2.21]{Kuronya3}) the following $2$-dimensional subset
\[
\Delta_E^{\mathrm{num}}(D):=\{(t,y) \in \mathbb{R}^2 \;|\; 0 \leq t \leq \mu_E(D), \; 0\leq y \leq q_X(P(D_t),E)\},
\]
where $P(D_t)$ is the positive part of the divisorial Zariski decomposition of $D_t$. More generally, if $D$ is pseudo-effective, we define 
\[
\Delta_E^{\mathrm{num}}(D):= \{(\mu_E(D),y) \in \mathbb{R}^2 \;|\; 0\leq y \leq q_X(P(D),E)\}.
\]
In particular, when $D$ is pseudo-effective and non-big, $\Delta_E^{\mathrm{num}}(D)$ is either a segment or a point.
\end{con}

\begin{rmk}\label{rmkboundary} 
When $D$ is a big divisor, we note that the boundary of $\Delta^{\mathrm{num}}_E(D)$ is defined by a piecewise linear function. Indeed, by \cite[Theorem 1.6]{Den} we know that $\mathrm{Big}(X)$ admits a locally finite decomposition into locally rational polyhedral subcones, called Boucksom-Zariski chambers.\ On each Boucksom-Zariski chamber the divisorial Zariski decomposition varies linearly, hence $\partial \Delta^{\mathrm{num}}_E(D)$ is piecewise linear.
\end{rmk}

\begin{prop}\label{thmden}
Let $D$ be any big $\mathbf{R}$-divisor $D$ on $X$ and $E$ a prime divisor not lying in $\mathbf{B}_{-}(D)$. Then $\Delta^{\mathrm{num}}_E(D)$ is a convex subset of $\mathbf{R}^2$ and its euclidean volume equals $q_X(P(D))/2$. In particular, we have the relations 
\[
2^nc_X(\mathrm{vol}_{\mathbf{R}^2}(\Delta^{\mathrm{num}}_E(D))^n=c_X(q_X(P(D)))^n=\mathrm{vol}(D)=(2n)! \cdot \mathrm{vol_{\mathbf{R}^n}}(\Delta(D)),
\] where $\Delta(D)$ is the Newton-Okounkov body of $D$, with respect to any admissible flag.
\begin{proof}
We first show that $\Delta^{\mathrm{num}}_E(D)$ is convex. Let $\mathrm{Big}^E(X)^+$ be the convex open subcone of $\mathrm{Big}(X)$ consisting of the big classes $\alpha$ such that $E\not\subset \mathbf{B}_{+}(\alpha)$. We claim that the function $q_X(-,E)$ is concave on $\mathrm{Big}^E(X)^+$. If we prove the claim we are done, because by Lemma \ref{lem1}, the class of any element of the form $D-tE$, with $t\in ]0,\mu_E(D)[$ lies in $\mathrm{Big}^E(X)^+$.
Let $D_1,D_2$ two big divisors whose classes lie in $\mathrm{Big}^E(X)^+$. It suffices to show that 
\[
q_X(P(D_1+D_2),E)\geq q_X(P(D_1),E)+q_X(P(D_2),E). 
\] 
We observe that $P(D_1+D_2)\geq P(D_1)+P(D_2)$, because $P(D_1+D_2)$ is the maximal $q_X$-nef subdivisor of $D_1+D_2$. Furthermore, as $D_1+D_2=P(D_1)+P(D_2)+N(D_1)+N(D_2)$, $P(D_1+D_2)$ is obtained by adding to $P(D_1)+P(D_2)$ something coming from $N(D_1)+N(D_2)$. But $E\not\in \mathrm{Supp}(N(D_1)+N(D_2))$, hence $q_X(P(D_1+D_2),E)\geq q_X(P(D_1),E)+q_X(P(D_2),E)$, and we are done. Now we compute the volume of $\Delta^{\mathrm{num}}_E(D)$. Using relation (\ref{relation}), we have 
\begin{equation}\label{relation1}
q_X(P(D_t),E)\int [P(D_t)]^{2n}=q_X(P(D_t))\int [P(D_t)]^{2n-1}[E], \text{ for any } t.
\end{equation}
By definition of $\Delta^{\mathrm{num}}_E(D)$,
\[
\mathrm{vol}_{\mathbf{R}^2}(\Delta^{\mathrm{num}}_E(D))=\int_0^{\mu_E(D)} q_X(P(D_t),E) dt.
\]
By Proposition \ref{prop:restvol} we have $\int [P(D_t)]^{2n-1}[E]=\mathrm{vol}_{X|E}(P(D_t))$. From equation (\ref{relation1}) we obtain
\begin{equation}\label{relation2}
\mathrm{vol}_{\mathbf{R}^2}(\Delta^{\mathrm{num}}_E(D))=\int_0^{\mu_E(D)}q_X(P(D_t),E) dt= \int_0^{\mu_E(D)} q_X(P(D_t))  \frac{\mathrm{vol}_{X|E}(P(D_t))}{\mathrm{vol}(P(D_t))} dt.
\end{equation}
By \cite[Corollary 4.25, item (iii)]{Mustata}, in the interval $]0,\mu_E(D)[$ it holds 
\begin{equation}\label{derivativeofvolume}
\frac{d}{dt}(\mathrm{vol}_{X}(P(D_t)))=-2n\cdot\mathrm{vol}_{X|E}(P(D_t)).
\end{equation}
Also, using the Fujiki relation and \cite[Proposition 4.12]{Bouck}, we have 
\[
q_X(P(D_t))=\frac{\left(\int [P(D_t)]^{2n}\right)^{1/n}}{(c_X)^{1/n}}=\frac{\mathrm{vol}(P(D_t))^{1/n}}{(c_X)^{1/n}}.
\]
Then equality (\ref{relation2}) becomes
\[
\mathrm{vol}_{\mathbf{R}^2}(\Delta^{\mathrm{num}}_E(D))=-\frac{1}{2n (c_X)^{1/n}}\int_0^{\mu_E(D)}- 2n \left(\mathrm{vol}(P(D_t))\right)^{\frac{1}{n}-1}\mathrm{vol}_{X|E}(P(D_t)) dt,
\]
and hence, by equality (\ref{derivativeofvolume}),
\[
\mathrm{vol}_{\mathbf{R}^2}(\Delta^{\mathrm{num}}_E(D))=-\frac{1}{2n(c_X)^{1/n}} \Bigg[\frac{\big(\mathrm{vol}(P(D_t))\big)^{1/n}}{1/n}\Bigg]^{\mu_E(D)}_{0}=\frac{1}{2 (c_X)^{1/n}}\mathrm{vol}(P(D))^{1/n},
\]
where the last equality in the above formula comes from the fact that $\mathrm{vol}(P(D_{\mu_E(D)}))=0$, because $P(D_{\mu_E(D)})$ is not big, by definition of the threshold $\mu_E(D)$. Using once again the Fujiki relation we obtain $ \mathrm{vol}_{\mathbf{R}^2}(\Delta^{\mathrm{num}}_E(D))=q_X(P(D))/2$, as wanted.
\end{proof}
\end{prop}

\begin{rmk}
Suppose that a prime divisor $E$ belongs to $\mathbb{B}_{-}(D)$, for some pseudo-effective divisor $D$ on $X$. Let $\nu_E(D)$ be the coefficient of $E$ in $N(D)$. Then we have 
\[
2 \cdot \mathrm{vol}_{\mathbf{R}^2}(\Delta^{\mathrm{num}}_E(D-\nu_E(D) E))=q_X(P(D)),
\]
 because $P(D)=P(D-\nu_E(D) E)$. Hence, up to subtracting $\nu_E(D) E$ from $D$, we can define $\Delta^{\mathrm{num}}_E(D)$, also if $E\subset \mathbf{B}_{-}(D)$. In particular, if $E$ is any prime divisor, we define $\Delta^{\mathrm{num}}_E(D):=\Delta^{\mathrm{num}}_E(D-\nu_E(D) E)$. Note that if $D$ is pseudo-effective and not big, then $\Delta^{\mathrm{num}}_E(D)$ has empty interior. 
\end{rmk}

We now recall the definition of Boucksom-Zariski chamber associated with a big divisor.

\begin{defn}
Let $D$ be any big $\mathbf{R}$-divisor on a projective IHS manifold $X$. The Boucksom-Zariski chamber associated with $D$ is the convex cone
\[
\Sigma_P:=\{[D] \in \mathrm{Big}(X)\;|\;\mathrm{Neg}_{q_X}(D)=\mathrm{Null}_{q_X}(P)\},
\]
where $\mathrm{Neg}_{q_X}(D)$ is the set of prime divisors in the support of $N(D)$, $\mathrm{Null}_{q_X}(P)$ is defined in Proposition \ref{prop1}, and $\mathrm{Neg}_{q_X}(D)$ is the set of the irreducible components of $N(D)$.
\end{defn}
With the following proposition, it turns out that $\Delta^{\mathrm{num}}_E(D)$ is a polygon.

\begin{prop}\label{prop:bodiesarepolygons}
Let $D$ be a big $\mathbf{R}$-divisor on $X$ and $E$ be a prime divisor. Then $\Delta^{\mathrm{num}}_E(D)$ is a polygon.
\begin{proof}
The proof goes as in \cite[Proposition 2.1]{Kuronya1}. By Remark \ref{rmkboundary} we only have to check that along the segment $[0,\mu_E(D)]$ we pass through a finite number of Boucksom-Zariski chambers. Without loss of generality, we can assume that $\nu_E(D)$ (the coefficient of $E$ in $N(D)$) is 0. Set $D':=D-\mu_E(D)E$ and $D'_s:=D'+sE$, for $s \in [0,\mu_E(D)]$. We claim that the function $s\mapsto N(D'_s)$ is non-increasing along $[0,\mu_E(D)]$, i.e.\ if $0 \leq s' < s \leq \mu_E(D)$, $N(D'_{s'})-N(D'_s)$ is effective. If we prove the claim we are done. Indeed, if the function $s\mapsto N(D'_s)$ is non-increasing, then for $s \in [0,\mu_E(D)]$ the negative part of the divisorial Zariski decomposition of $D'_s$ will be supported on some of the irreducible components of $N(D'_0)=N(D')$. In particular, along the segment $[0,\mu_E(D)]$ we will pass through finitely many Boucksom-Zariski chambers (see \cite[Definition 4.14, Corollary 4.17]{Den}). Now, we have
\[
P(D'_{s'})=D'_{s'}-N(D'_{s'})=(D'_s-(s-s')E)-N(D'_{s'}).
\]
As $P(D'_{s'})$ is $q_X$-nef and $P(D'_s)$ is the maximal $q_X$-nef subdivisor of $D'_s$, 
\[
P(D'_s)-P(D'_{s'})=(s-s')E+N(D'_{s'})-N(D'_{s})\]
is effective. Then we are left to show that $E$ is not contained in the support of $N(D'_{s})$, for any $s \in [0,\mu_E(D)]$. Indeed, if $\mathrm{Supp}(E) \subset \mathrm{Supp}(N(D'_{s}))$ for some $s \in [0,\mu_E(D)]$, for any $\lambda>0$ the divisorial Zariski decomposition of $D'_{s+\lambda}$ would be $D'_{s+\lambda}=P(D'_s)+N(D'_s)+\lambda E$, with $P(D'_s)=P(D'_{s+\lambda})$ and $N(D'_{s+\lambda})= N(D'_s)+\lambda E$. This would imply $\mathrm{Supp}(E) \subset \mathrm{Supp}(N(D'_{\mu_E(D)}))=\mathrm{Supp}(N(D))$, and this is absurd.
\end{proof}
\end{prop}

\begin{rmk}
Let $D$ be a big $\mathbf{Q}$-divisor on $X$, $E$ a prime divisor on $X$ and $\nu_E(D)$ the coefficient of $E$ in $N(D)$. Then all the vertices of the polygon $\Delta^{\mathrm{num}}_E(D)$ contained in the set $\{[\nu_E(D), \mu_E(D)[ \times \mathbb{R}\}$ have rational coordinates. Indeed, $\mathrm{Big}(X)$ admits a decomposition into locally rational polyhedral subcones, hence all the break-points with the first coordinate lying in $[\nu_E(D), \mu_E(D)[$ are rational.
\end{rmk}

\begin{prop}
Let $D$ be a big $\mathbf{Q}$-divisor on $X$ and $E$ a prime divisor. Then $\mu_E(D)$ is algebraic.
\end{prop}
\begin{proof}
Without loss of generality, we can assume that $\Delta^{\mathrm{num}}_E(D)$ is a triangle. We know by Theorem \ref{thmden} that the volume of $\Delta^{\mathrm{num}}_E(D)$ equals $q_X(P(D))/2$, which is a rational number. The last vertex of $\Delta^{\mathrm{num}}_E(D)$ has coordinates $\left(\mu_E(D),q_X(P(D)-\mu_E(D)E,E)\right)$. Then $\mu_E(D)$ is a solution of a quadratic equation with rational coefficients, hence it is algebraic.
\end{proof}

\begin{proof1}
The theorem is a consequence of Proposition \ref{thmden} and Proposition \ref{prop:bodiesarepolygons}.
\end{proof1}

 The following proposition is analogous to \cite[Proposition 11]{AKL14}.

\begin{prop}
Let $D$ be any big $\mathbf{Q}$-divisor on $X$. There exists a prime divisor $E$ with respect to which $\Delta_E^{\mathrm{num}}(D)$ is a rational simplex.
\end{prop}
\begin{proof}
Let $D=P(D)+N(D)$ be the divisorial Zariski decomposition of $D$, with $N(D)=\sum_i a_iE_i$, where $a_i\geq 0$ and $E_i$ is a prime divisor, for any $i$. We know that $\mathbf{B}_{\mathrm{div}}(D)=\cup_i \mathrm{Supp}(E_i)$, where $\mathbf{B}_{\mathrm{div}}(D)$ denotes the divisorial stable base locus of $D$. When $k\gg 0$ is sufficiently divisible, $\mathbf{B}_{\mathrm{div}}(kD)$ coincides with the fixed part of the linear series $|kD|$. Now, consider $D'_k=k(D-\sum_ia_iE_i)$, with $k\gg 0$ sufficiently divisible. Then $|D'_k|$ has no fixed components and by the Bertini Theorem, the general member of $|D'_k|$ is an irreducible, reduced divisor. Let $E$ be such a member of $|D'_k|$. Consider $D_t=D-tE$. We have $D_t \equiv D-tk(D-\sum_ia_iE_i)=(1-tk)D+tk \sum_i a_iE_i$ and we observe that $\mu_E(D)=1/k$. Indeed, if $0 \leq t \leq 1/k$, $D_t$ is pseudo-effective (actually, for $0 \leq t < 1/k$, it is big).  We have $q_X(P(D),D_t)=(1-tk)q_X(P(D))$. If for some value of $t$ greater than $1/k$ $D_t$ was effective, we would have $q_X(P(D),D_t)<0$, which is absurd, as $P(D)$ is $q_X$-nef. Then $\mu_E(D)=1/k$. We conclude that, for any $t \in [0,1/k[$, $[D_t]$ lies in a single Boucksom-Zariski chamber and $\Delta_E^{\mathrm{num}}(D)$ is a rational simplex.
\end{proof}

We now globalise Construction \ref{construction1} and prove Theorem \ref{thm:2}.

\begin{proof2}
If $\zeta$ is any class in $N^1(X)_{\mathbf{R}}$ and $t$ is any real number, we set $\zeta_t:=\zeta-tE$. We define \[
\Delta^{\mathrm{num}}_E(X):=\big\{(\zeta,t,y)\in \mathrm{Eff}(X)\times \mathbb{R}^2 \;|\; \nu_E(\zeta) \leq t \leq \mu_E(\zeta), \; 0 \leq y \leq q_X(P(\zeta_t),E) \big\},\]
 where $\nu_E(\zeta)$ is the coefficient of $E$ in $N(\zeta)$. If $(\zeta,t,y) \in \Delta^{\mathrm{num}}_E(X)$, clearly $\lambda \cdot (\zeta,t,y) \in \Delta^{\mathrm{num}}_E(X)$, for any positive real number $\lambda$, and so $\Delta^{\mathrm{num}}_E(X)$ is a cone. We have to check that $\Delta^{\mathrm{num}}_E(X)$ is convex, namely, if $(\zeta',t',y'),  (\zeta',t',y')$ belong to $\Delta^{\mathrm{num}}_E(X)$, we have to show that also $(\zeta+\zeta',t+t',y+y')$ belongs to $\Delta^{\mathrm{num}}_E(X)$. We start by observing that, if $\zeta-\mu_E(\zeta)E$ and $\zeta'-\mu_E(\zeta')E$ lie on the same extremal face of $\overline{\mathrm{Eff}(X)}$, the threshold $\mu_E(\zeta+\zeta')$ is equal to $\mu_E(\zeta)+\mu_E(\zeta')$.
If $\zeta-\mu_E(\zeta)E$ and $\zeta'-\mu_E(\zeta')E$ lie on different extremal faces, the divisor $\zeta+\zeta'-(\mu_E(\zeta)+\mu_E(\zeta'))E$ is big. It follows that  $\mu_E(\zeta+\zeta')\geq \mu_E(\zeta)+\mu_E(\zeta')$.
In particular, we obtain $\nu_E(\zeta+\zeta')\leq t+t'\leq \mu_E(\zeta+\zeta')$. Now, we have to prove that $q_X(P((\zeta+\zeta')_{t+t'}),E)\geq y+y'$ and that the equality $q_X(P((\zeta+\zeta')_{t+t'}),E)=0$ could occur only if $q_X(P(\zeta_t),E)=q_X(P({\zeta'}_{t'}),E)=0$. By definition of $\Delta_E^{\mathrm{num}}(X)$, we have \[E\not\subset \mathbf{B}_{-}(\zeta_t) \cup \mathbf{B}_{-}(\zeta'_{t'})=\mathrm{Supp}(N(\zeta_t))\cup \mathrm{Supp}(N({\zeta'}_{t'})).\] Furthermore $(\zeta+\zeta')_{t+t'}=\zeta_t+\zeta_{t'}$, so that \[P((\zeta+\zeta')_{t+t'})=P(\zeta_t+{\zeta'}_{t'})=P(\zeta_t)+P({\zeta'}_{t'})+N,\] where $N$ is an effective divisor not containing $E$ in its support. Then $q_X(P(\zeta_t+{\zeta'}_{t'}),E)\geq 0$, because $q_X(P(\zeta_t),E)\geq 0$, $q_X(P({\zeta'}_{t'}),E))\geq 0$, since $P(\zeta_t)$ and $P({\zeta'})$ are $q_X$-nef, and $q_X(N,E)\geq 0$, by \cite[Proposition 4.2, item (ii)]{Bouck}), as $E\not\subset \mathrm{Supp}(N)$. Clearly, $q_X(P(\zeta_t+{\zeta'}_{t'}),E)=0$ could occur only if $E$ belongs to $\mathbf{B}_{+}(\zeta_t)$ and $\mathbf{B}_{+}(\zeta'_{t'})$. We conclude that $q_X(P(\zeta_t+{\zeta'}_{t'}),E) \geq y+y' \geq 0$, thus $(\zeta+\zeta',t+t',y+y')\in \Delta_E^{\mathrm{num}}(\zeta+\zeta')$, and we are done.
\end{proof2}

Using Theorem \ref{thm:2}, we recover the log-concavity relation for the volume function on a projective IHS manifold.

\begin{cor}
Let $\zeta,\zeta'$ be two big divisor classes, and $E$ a prime divisor on $X$. We have $\Delta_E^{\mathrm{num}}(\zeta)+\Delta_E^{\mathrm{num}}(\zeta') \subset \Delta_E^{\mathrm{num}}(\zeta+\zeta')$, and hence we obtain the inequality 
\begin{equation}\label{logconcavity}
q_X(P(\zeta+\zeta'))^{1/2}\geq q_X(P(\zeta))^{1/2}+q_X(P(\zeta'))^{1/2}.
\end{equation}
Using the Fujiki relation, we recover the log-concavity relation for the volume function, i.e.\
\begin{equation}\label{logconcavity1}
\mathrm{vol}_X(\zeta+\zeta')^{\frac{1}{2n}}\geq \mathrm{vol_X}(\zeta)^{\frac{1}{2n}}+\mathrm{vol}_X(\zeta')^{\frac{1}{2n}}.
\end{equation}
\end{cor}
\begin{proof}
Inequality (\ref{logconcavity}) is obtained by using the Brunn-Minkowski Theorem, while inequality (\ref{logconcavity1}) is obtained from (\ref{logconcavity}) using the Fujiki relation.
\end{proof}

\begin{rmk}
Notice that inequality (\ref{logconcavity1}) was obtained by Lazarsfeld for any projective variety (see \cite[Theorem 11.4.9]{Laz}). His proof makes use of Fujita's approximation Theorem (see \cite[Theorem 11.4.4]{Laz}). Hence, inequality (\ref{logconcavity}) can be directly obtained using the Fujiki relation and Lazarsfeld's result. Our proof (of both inequalities) for IHS manifolds does not make use of Fujita's approximation Theorem.
\end{rmk}

Following \cite{SS16}, we now show that, under certain assumptions, the cone $\Delta_E^{\mathrm{num}}(X)$ is rational polyhedral, providing a set of rational generators.

\begin{prop}\label{generators}
Let $E$ be a non-exceptional prime divisor. Suppose that $\mathrm{Eff}(X)$ is rational polyhedral, as well as the closure of any Boucksom-Zariski chamber of $\mathrm{Big}(X)$. Let $\{D_i\}_i$ be a set of rational generators of the extremal rays of the closure of the Boucksom-Zariski chambers. Then $\Delta_E^{\mathrm{num}}(X)$ is rational polyhedral, generated by the sets $T:=\{([D_i],0,q_X(P(D_i),E)\}_i$, $T':=\{([D_i],0,0)\}_i$, and the element $([E], 1,0)$.
\begin{proof}
We first show that the cone generated by $T$, $T'$ and $([E], 1,0)$ is contained in $\Delta_E^{\mathrm{num}}(X)$. It suffices to show that $([E], 1,0) \in \Delta_E^{\mathrm{num}}(X)$ and \[ ([D_i],0,q_X(P(D_i),E),\;([D_i],0,0) \in \Delta_E^{\mathrm{num}}(X),\] for any $i$. We observe that $\Delta_E(E)$ contains the point $(1,0)$, because $\mu_E(E)=1$, hence $([E], 1,0) \in \Delta_E^{\mathrm{num}}(X)$. The elements of $T$ and $T'$ belong to $\in \Delta_E^{\mathrm{num}}(X)$ by definition of $\Delta_E^{\mathrm{num}}(D_i)$ and because $E \not \subset \mathbf{B}_-(D_i)$, as $E$ is not exceptional. We now show that $\Delta_E^{\mathrm{num}}(X)$ is contained in the cone generated by $T,T'$ and $(E,1,0)$. Let $([D],a,b)$ be an element of $\Delta_E^{\mathrm{num}}(X)$. Set $D_a=D-aE$ and let $D_a=P(D_a)+N(D_a)$ be the divisorial Zariski decomposition of $D_a$. There exists a big and $q_X$-nef divisor $P$ such that $D_a$ belongs to the closure $\overline{\Sigma_P}$ of the Boucksom-Zariski chamber $\Sigma_P$ associated with $P$. Let $\{D_{t_j}\}_j$ be the generators of $\overline{\Sigma_P}$. Then $D_a=\sum_j a_j D_{t_j}$, with $a_j\geq 0$, for any $j$. In particular, $P(D_a)=\sum_j a_j P(D_{t_j})$ and
\[
0\leq b \leq q_X(P(D_a),E)=\sum_ja_jq_X(P(D_j),E).
\]
Then there exists $c\in[0,1]$ such that $b=c \sum_ja_jq_X(P(D_j),E)$.  Some easy calculations show that
\[
([D_a],0,b)=c\sum_ja_j([D_j],0,q_X(P(D_j),E))+(1-c)\sum_ja_j([D_j],0,0).
\]
If we add $(a[E],a,0)$ to both members of the equality above, we obtain
\[
([D],a,b))=c\sum_ja_j([D_j],0,q_X(P(D_j),E))+(1-c)\sum_ja_j([D_j],0,0)+a([E],1,0),
\]
which concludes the proof.
\end{proof}
\end{prop}

\section{Minkowski decompositions}
In this section, we prove Theorem \ref{thm:3}. Before proving it, we need the following proposition, which tells us that, up to a translation, the polygon $\Delta_E^{\mathrm{num}}(D)$ of any big $\mathbf{R}$-divisor $D$ with respect to any prime divisor $E$ is determined by the positive part of $D$.

\begin{prop}\label{translation}
Let $D$ be a big $\mathbf{R}$-divisor and $E$ a prime divisor on $X$. Moreover, let $\nu_E(D)$ be the coefficient of $E$ in $N(D)$. Then $\Delta_E^{\mathrm{num}}(D)=\Delta_E^{\mathrm{num}}(P(D))+\overrightarrow{(\nu_E(D),0)}$.
\begin{proof}
Without loss of generality, we can assume that $\nu_E(D)=0$. Hence, we have to show that $\Delta_E^{\mathrm{num}}(D)=\Delta_E^{\mathrm{num}}(P(D))$, and it suffices to show that $P(D-tE)=P(P(D)-tE)$, for any $t \in [0,\mu_E(D)[$. We start by observing that $\mu_E(D)=\mu_E(P(D))$. Indeed, by Proposition \ref{prop:bodiesarepolygons}, $N(D-tE)-N(D)$ is effective, which implies that $P(D)-tE$ is big, for any $t \in [0,\mu_E(D)[$. This implies that $\mu_E(P(D))\geq \mu_E(D)$, and if this inequality is strict, we would have that $D-\mu_E(P(D))E$ is pseudo-effective, contradicting the maximality of $\mu_E(D)$. Then $\mu_E(P(D))= \mu_E(D)$. We now show that $P(P(D)-tE)$ is equal to $P(D-tE)$. We can write \[D-tE=P(D)-tE+N(D)=P(P(D)-tE)+N(P(D)-tE)+N(D).\] If $P(P(D)-tE)\neq P(D-tE)$, there must be an irreducible component $E'$ of $N(D)$ such that $q_X(P(P(D)-tE),E')>0$. Furthermore, $P(D-tE)=P(P(D)-tE)+N$, where $N$ is an effective divisor supported by some of the irreducible components of $N(P(D)-tE)$, by Proposition \ref{prop:bodiesarepolygons}. Since $\mathrm{Supp}(E')\not\subset \mathrm{Supp}(N)$, we have $q_X(P(D-tE),E')>0$, and this gives a contradiction. We conclude then that $P(P(D)-tE)=P(D-tE)$, for any $t\in [0,\mu_E(D)[$. 
\end{proof}
\end{prop}

\subsection{Constructing a Minkowski basis}\label{Minkowskibasis}

\begin{defn}
Given a prime divisor $E$, a Minkowski basis with respect to $E$ is a set $\Omega=\Omega(E)$ of movable $\mathbf{Q}$-divisors (i.e. divisors whose classes lie in $\mathrm{Mov}(X)$) such that for any effective $\mathbf{Q}$-divisor $D$ there exist rational numbers $\{\alpha_P(D)\}_{P\in \Omega}$ such that $D=\sum_P\alpha_P(D)P$ and $\Delta_E^{\mathrm{num}}(D)=\sum_P\alpha_P(D)\Delta_E^{\mathrm{num}}(P)$. 
\end{defn}

In what follows, we provide an algorithmic way to construct a Minkowski basis $\Omega$ with respect to a big prime divisor $E$, when $\mathrm{Eff}(X)$ is rational polyhedral. 
\vspace{0.2cm}

We start by associating with any Boucksom-Zariski chamber of $\mathrm{Big}(X)$ an element of $\Omega$. To do so, let $\Sigma=\Sigma_P$ be the Boucksom-Zariski chamber associated with an integral divisor $P \in \mathrm{Mov}(X)\cap \mathrm{Big}(X)$. 
\medskip

The negative part of the divisorial Zariski decomposition of any element in $\Sigma$ is supported on a set $S=\{E_1,\dots,E_k\}$ of prime exceptional divisors. Let $E$ be a big prime divisor, and consider the vector subspace $W_{S\cup E}$ of $N^1(X)_{\mathbf{R}}$ generated by $S$ and $E$. Then $\left(W_{S\cup E}\right)\cap S^{\perp}$ is a one dimensional vector subspace of $N^1(X)_{\mathbf{R}}$ (because $S\cup E$ spans a $(k+1)$-dimensional vector subspace of $N^1(X)_{\mathbf{R}}$, and $S^{\perp}$ is a $(\rho(X)-k)$-dimensional vector subspace of $N^1(X)_{\mathbf{R}}$), spanned by the class of some $\mathbf{Q}$-divisor of the form $D=xE+\sum_{i=1}^kx_iE_i$. We claim that $x$ and the $x_i$'s are all of the same sign, so that either $D$ or $-D$ lies in $\mathrm{Mov}(X)$. Let $G$ be the Gram-matrix of $S$. By \cite[Lemma 4.1]{Bau} the entries of $S^{-1}$ are non-positive. Then, for a fixed $x \neq 0$, the solution of the linear system
\[
S(x_1,\dots,x_k)^{T}=-x\left(q_X(E,E_1),\dots,q_X(E,E_k)\right)^{T}
\]
 has the same sign of that of $x$, because $q_X(E,E_j)\geq 0$, for any $j=1,\dots,k$. Choosing $x$ rational and positive, we obtain a rational generator $D_{\Sigma}$ of $\left(W_{S\cup E}\right)\cap S^{\perp}$ which lies in $\mathrm{Mov}(X)$. Up to clearing the denominators of $D_{\Sigma}=xE+\sum_{i=1}^kx_iE_i$, we can assume that $D_{\Sigma}$ is integral. We observe that $D_{\Sigma}$ is big, because $E$ is, and $\sum_{i=1}^kx_iE_i$ is effective.
\vspace{0.2cm}
 
We can now describe a Minkowski basis $\Omega$. For any Boucksom-Zariski chamber $\Sigma$ we add to $\Omega$ the integral divisor $D_{\Sigma}$ constructed above, and for any extremal ray $R$ of $\mathrm{Mov}(X)$ of BBF square 0, we add to $\Omega$ an integral generator of $R$.
\vspace{0.2cm} 

We observe that since we are assuming that $\mathrm{Eff}(X)$ is rational polyhedral (and so also $\mathrm{Mov}(X)$ is), $\Omega$ is finite. Furthermore, $\mu_E(D_{\Sigma})=x$, and, for any $0<t<\mu_E(D)$, $D_t$ belongs to $\Sigma$, by construction.
\vspace{0.2cm}

We are now ready to prove Theorem \ref{thm:3}. Recall that a set $S$ of prime exceptional divisors on a projective IHS manifold $X$ is called an exceptional block if either $S=\emptyset$, or $S=\{E_1,\dots,E_k\}$ and the Gram matrix $(q_X(E_i,E_j))_{i,j}$ is negative-definite.

\begin{proof3}
Without loss of generality, by Proposition \ref{translation}, we can assume that $D=P(D)$, in particular that $D$ is movable. If $D$ is not big, $q_X(D)=0$, and so $\lambda D= D'$ for some $D' \in \Omega$, $\lambda \in \mathbf{Q}$. Then $\Delta_E^{\mathrm{num}}(D)=\frac{1}{\lambda}\Delta_E^{\mathrm{num}}(D')$. Otherwise, let $\Sigma$ be the Boucksom-Zariski chamber associated with the exceptional block $\mathrm{Null}_{q_X}(D)$ (see \cite[Corollary 4.17]{Den}), and $D_{\Sigma}$ the element of $\Omega$ attached to $D$. Define the movable threshold
\[
\tau=\tau(D):=\mathrm{sup}\left\{t\geq 0 \; | \; D-t D_{\Sigma} \text{ is movable}\right\}
\]
and set $D':=D-\tau D_{\Sigma}$ and observe that $D'$ belongs to $\mathrm{Face}(D):=\mathrm{Null}_{q_X}(D)^{\perp} \cap \mathrm{Mov}(X)$. Moreover, we claim that $\tau$ is a rational number. To show it, we first observe that, either $q_X(D')=0$, or $D'$ is big and $\mathrm{Null}_{q_X}(D)\subsetneq \mathrm{Null}_{q_X}(D')$. Indeed, if $D'$ is big and $\mathrm{Null}_{q_X}(D)= \mathrm{Null}_{q_X}(D')$, there would exist $0<\epsilon \ll 1$ such that $D-(\tau+\epsilon)D_{\Sigma}$ lies in $\mathrm{Mov}(X)$, by \cite[Proposition 4.28]{Den}, and this would contradict the maximality of $\tau$.

If $D'=0$, $D=\tau D_{\Sigma}$, and so $\tau$ is rational. If $D' \neq 0$ and $q_X(D')=0$, we have $D'=aM$, where $M$ is a $\mathbf{Q}$-divisor and $a$ is a positive real number. Then $D=aM+\tau D_{\Sigma}$, and so
\[
q_X(D,M)=\tau q_X(D_{\Sigma},M)>0.
\]
This forces $\tau$ to be rational. 

If $D'$ is big, pick a prime divisor $E$ lying in $\mathrm{Null}_{q_X}(D') \setminus \mathrm{Null}_{q_X}(D)$. Then we have
\[
q_X(D,E)=\tau q_X(D_{\Sigma},E),
\] 
and so $\tau$ is rational.

Now, if $D'=0$, we have $\Delta_E^{\mathrm{num}}(D)=\tau \Delta_E^{\mathrm{num}}(D_{\Sigma})$, and  we are done. Otherwise we claim that
\begin{equation}\label{Minkowskidecomp}
\Delta_E^{\mathrm{num}}(D)=\Delta_E^{\mathrm{num}}(D')+\tau\Delta_E^{\mathrm{num}}(D_{\Sigma}).
\end{equation}
To show equality (\ref{Minkowskidecomp}) we first observe that if $0\leq t \leq \mu_E(\tau D_{\Sigma})$, the divisor $D-tE=D'+(\tau D_{\Sigma} -tE)$ is effective. Indeed, $D'$ is effective, since we are assuming that $\mathrm{Eff}(X)$ is rational polyhedral, and \[
\tau D_{\Sigma}-tE=(\tau x-t)E+\sum_{i}x_iE_i \] is effective, as $\mu_E(\tau D_{\Sigma})=\tau x$. Also, note that $\mu_E(D)\geq \mu_E(\tau D_{\Sigma})$. Then, for any $0\leq t \leq \mu_E(\tau D_{\Sigma})$, we have
\[
P(D-tE)=P(D'+\tau D_{\Sigma}-tE)=P(D')+P(\tau D_{\Sigma}-tE).
\]
Now, let $t$ be such that $\mu_E(\tau D_{\Sigma})\leq t \leq \mu_E(D)$. If $\mu_E(\tau D_{\Sigma})=\mu_E(D)$, $D'$ is not big, and we are done. Otherwise $D'$ is big, and, arguing as in Proposition \ref{prop:bodiesarepolygons}, we see that $D'-(t-\mu_E(\tau D_{\Sigma}))E$ is big for any $t \in [\mu_E(\tau D_{\Sigma}),\mu_E(D)[$, and $N(D-tE)\geq \sum_ix_iE_i$. In particular, 
\[
P(D-tE)=P(D'-(t-\mu_E(\tau D_{\Sigma}))E),
\] 
for any $t$ in $[\mu_E(\tau D_{\Sigma}),\mu_E(D)]$. Resuming, we have

\[
P(D_t)=\begin{cases}
P(D')+P(\tau D_{\Sigma}-tE), & \text{ if } 0\leq t \leq \mu_E(\tau D_{\Sigma}),\\
P(D'-(t-\mu_E(\tau D_{\Sigma}))E), & \text{ if } \mu_E(\tau D_{\Sigma})\leq t \leq \mu_E(D),
\end{cases}
\]

and so 

\[
P(D_t)=\begin{cases}
P(D')+(\tau x -t)E , & \text{ if } 0\leq t \leq \tau x,\\
P(D'-(t-\mu_E(\tau D_{\Sigma}))E), & \text{ if } \tau x\leq t \leq \mu_E(D).
\end{cases}
\]
Then
\[
q_X(P(D_t),E)=\begin{cases}
q_X(P(D'),E)+q_X(P(\tau D_{\Sigma}-tE),E), & \text{ if } 0\leq t \leq \mu_E(\tau D_{\Sigma}),\\
q_X(P(D'-(t-\mu_E(\tau D_{\Sigma}))E),E), & \text{ if } \mu_E(\tau D_{\Sigma})\leq t \leq \mu_E(D),
\end{cases}
\]
 that means $\Delta_E^{\mathrm{num}}(D)=\Delta_E^{\mathrm{num}}(D')+\tau\Delta_E^{\mathrm{num}}(D_{\Sigma})$.
 \vspace{0.1cm}
 
 Then we can consider the Boucksom-Zariski chamber associated with $D'$, which is different to that associated with $D$, by \cite[Corollary 4.17]{Den}. We can repeat the above procedure to decompose $\Delta^{\mathrm{num}}_E(D')$, as we did for $\Delta^{\mathrm{num}}_E(D)$. This can be done finitely many times because at any step the dimension of $\mathrm{Face}(D')=\mathrm{Null}_{q_X}(D')^{\perp}\cap \mathrm{Mov}(X)$ decreases of at least $1$, and the dimension of $\mathrm{Face}(B)$, where $B$ is any big and movable divisor, is at most $\rho(X)$.
 
\end{proof3}

If we allow a Minkowski basis to be infinite and we suppose that $\overline{\mathrm{Eff}(X)}=\mathrm{Eff}(X)$ (which does not imply that $\mathrm{Eff}(X)$ is polyhedral), Theorem \ref{thm:3} holds in more generality. In particular, we have the following corollary.

\begin{cor}\label{finalcor}
Let $E$ be any big prime divisor on $X$ and suppose that $\overline{\mathrm{Eff}(X)}=\mathrm{Eff}(X)$. Then there exists a set $\Omega=\Omega(E)$ of $\mathbf{Q}$-divisors whose classes lie in $\overline{\mathrm{Mov}(X)}$, such that, if $D$ is any big $\mathbf{Q}$-divisor, there exist some rational numbers $\{\alpha_P(D)\}_{P\in \Omega}$ such that $P(D)=\sum_{P\in \Omega} \alpha_P(D)P$ and, up to a translation, $\Delta_E^{\mathrm{num}}(D)=\sum_{P\in \Omega}\alpha_P(D)\Delta_E^{\mathrm{num}}(P)$. In particular, for any big $\mathbf{Q}$-divisor $D$, the polygon $\Delta_E^{\mathrm{num}}(D)$ is rational.
\end{cor}
\begin{proof}
The proof is the same as that of Theorem \ref{thm:3}.
\end{proof}

\begin{rmk}
 Theorem \ref{thm:3} and Corollary \ref{finalcor} hold  also for any big $\mathbf{R}$-divisor $D$ on $X$. The only difference is that the coefficients $\alpha_P(D)$ appearing in the Minkowski decomposition might not be rational.
\end{rmk}

\section{Examples}\label{examples}
We start by computing some polygons of Newton-Okounkov type on Hilbert schemes of points on K3 surfaces.

\begin{ex}
Let $X$ be $\mathrm{Hilb}^2(S)$, where $S$ is a projective K3 surface with $\mathrm{Pic}(S)=\mathbf{Z}\cdot H_S$ and $H_S^2=2$. Now, let $E'$ be a general member of the linear system $|H-\delta|$, and $E$ the exceptional divisor of the Hilbert-Chow morphism $S^{[2]}\to S^{(2)}$. 
We compute the polygons $\Delta_{E'}^{\mathrm{num}}(3H-E)$, $\Delta_E^{\mathrm{num}}(H)$. Consider $\Delta_{E'}^{\mathrm{num}}(3H-E)$. We have $\mu_{E'}(3H-E)=3$, because $3H-E-3(H-\frac{1}{2}E)=\frac{3}{2}E$ is effective, but not big. Then
\[
P(3H-E-tE')=\begin{cases}
3H-E-tE', & \text{ if } 0\leq t \leq 2,\\
(3-t)H, & \text{ if } 2 \leq t \leq 3.
\end{cases}
\]
It follows that 
\[ \Delta_{E'}^{\mathrm{num}}(3H-E)=\left\{ (t,y) \in \mathbf{R}^2 \;|\; 0 \leq t \leq 2, \; 0\leq y \leq 2 \text{ or } 2 \leq t \leq 3, \; 0\leq y \leq 6-2t \right\},\]
which is a trapezium. In this case, the euclidean volume of $\Delta_{E'}^{\mathrm{num}}(3H-E)$ is $q_X(3H-E)/2=5$.
\begin{figure}
\begin{tikzpicture}
    \draw[gray!50, thin, step=0.5] (0,0) grid (4,3);
    \draw[very thick,->] (0,0) -- (4.4,0) node[right] {$t$};
    \draw[very thick,->] (0,0) -- (0,3.4) node[above] {$y$};
    \foreach \x in {1,...,4} \draw (\x,0.05) -- (\x,-0.05) node[below] {\tiny\x};
    \foreach \y in {1,...,3} \draw (-0.05,\y) -- (0.05,\y) node[left] {\tiny\y};

    \fill[green,opacity=0.5] (0,2) -- (2,2) -- (3,0) -- (0,0);

    \draw (0,2) -- (2,2);
    \draw (2,2) -- (3,0);
    
    \node at (1.2,1) {$\scriptstyle\Delta_{E'}^{\mathrm{num}}(3H-E)$};
\end{tikzpicture}
   \caption{The polygon $\scriptstyle\Delta_{E'}^{\mathrm{num}}(3H-E)$.}
\end{figure}
We now compute $\Delta_E^{\mathrm{num}}(H)$. We have $\mu_E(H)=1/2$, because $P(H_t)=H_t=H-tE$ for any $t\in[0,1/2]$, and $H-\frac{1}{2}E=H-\delta$ is not big, but pseudo-effective. Then
\[
\Delta_E^{\mathrm{num}}(H)=\left\{ (t,y) \in \mathbf{R}^2 \;|\; 0 \leq t \leq 1/2, \; 0\leq y \leq 8t \right\}
\]
is a triangle. The volume of $\Delta_E^{\mathrm{num}}(H)$ is given by $q_X(H)/2=1$.
\medskip

We now provide a set of generators $G$ for $\Delta_E^{\mathrm{num}}(X)$. We have 2 Boucksom-Zariski chambers, $\Sigma_H$ and $\Sigma_M$, where $M$ is any movable divisor. Furhermore, $\overline{\Sigma_M}=\langle H-\delta,H \rangle$, $\overline{\Sigma_H}=\langle H,\delta\rangle$. In particular, $\mathrm{Eff}(X)$ is rational polyhedral, as well as the Boucksom-Zariski chambers. Then we can use Proposition \ref{generators} and pick \[G=\{([2H-E],0,0),([H],0,0),([E],0,0),([2H-E],0,8),([E],1,0)\}.\] 
\end{ex}

We now give an example of a K3 surface $S$ such that the classical Newton-Okounkov body $\Delta_{F_{\bullet}}(D)$ of any big $\mathbf{Q}$-divisor $D$ with respect to a general admissible flag $F_{\bullet}=\left\{\{pt\}\subset C \subset S\right\})$ is a rational polygon. In this case $\Delta_{F_{\bullet}}(D)=\Delta_C^{\mathrm{num}}(D)$, for any big $\mathbf{R}$-divisor $D$.

\begin{ex}
Let $S$ be a K3 surface such that $\mathrm{Pic}(S)$ is isomorphic to $\mathbf{Z}^3$, with intersection form given by the matrix
\begin{equation*}
M = 
\begin{pmatrix}
0 & 1 & 1 \\
1 & -2 & 0 \\
1 & 0 & -2 
\end{pmatrix}
\end{equation*}
One can show (see \cite{Totaro}) that $S$ has a unique elliptic fibration  $S\to \mathbf{P}^1 $, induced by some nef line bundle $L$, and infinitely many smooth rational curves, each of which is a section of $S\to \mathbf{P}^1$. The ray spanned by $L$ in $\mathrm{Nef}(S)$ is extremal, and is the only extremal ray of square $0$ of the nef cone. Let $C\subset S$ be a big curve. Applying Corollary \ref{finalcor}, we find an \textit{infinite} set of nef $\mathbf{Q}$-divisors $\Omega$, such that any nef $\mathbf{Q}$-divisor is a linear combination (with coefficients in $\mathbf{Q}$) of elements in $\Omega$, and any polygon of Newton-Okounkov type $\Delta_C^{\mathrm{num}}(D)$ is, up to translation, the Minkowski sum (with coefficients in $\mathbf{Q}$) of some of the polygons ${\{\Delta_C^{\mathrm{num}}(P)\}}_{P \in \Omega}$. If we choose $\{pt\}$ to be a general smooth point of $C$ and set
\[
F_{\bullet}:=\{\{pt\} \subset C\subset S\},
\]
we clearly have $\Delta_C^{\mathrm{num}}(D)=\Delta_{F_{\bullet}}(D)$, for any big $\mathbf{R}$-divisor $D$. In particular, since any of the $\Delta_C^{\mathrm{num}}(D)$ is rational whenever $D$ is rational, also any of the $\Delta_{F_{\bullet}}(D)$ is.
\end{ex}

The following example shows that the bodies $\Delta_{C}^{\mathrm{num}}(D)$  on a projective K3 surface $S$ are in some sense preserved when passing to the Hilbert scheme of points $S^{[n]}$.
\begin{ex}
Let $S$ be a projective K3 surface and $F_{\bullet}=\left\{\{pt\}\subset C \subset S \right\}$ a general admissible flag. Let $D$ be any big $\mathbf{R}$-divisor on $S$. By the generality of $F_{\bullet}$ and \cite[Theorem 6.4]{Mustata}, the classical Newton-Okounkov body $\Delta_{F_{\bullet}}(D)$ is 
\[
\Delta_{F_{\bullet}}(D)=\Delta_C^{\mathrm{num}}(D)=\left\{(t,y) \in \mathbf{R}^2 \; | 0\leq t \leq \mu_C(D), \; 0\leq y \leq P(D_t)\cdot C\right\}.
\]
Now, let $X=S^{[n]}$ be the Hilbert scheme of $n$ points on $S$. Any irreducible curve $C'$ on $S$ is canonically associated with a prime divisor on $X$. This correspondence induces a homomorphism $\iota \colon \mathrm{Div}_{\mathbf{R}} (S) \to  \mathrm{Div}_{\mathbf{R}}(X)$, which in turn induces the usual injective group homomorphism $\iota' \colon \mathrm{Pic}(S) \hookrightarrow \mathrm{Pic}(X)$, which embeds $\mathrm{Pic}(S)$ in $\mathrm{Pic}(X)$ as a sublattice. In particular, ${C'}^2=q_X\left(\iota'([C'])\right)$, and to any exceptional block on $S$ (i.e.\ the Gram matrix of the elements of $S$, with respect to $q_X$,  is negative definite) corresponds an exceptional block on the Hilbert scheme $X$. Furthermore, $\iota$ preserves the divisorial Zariski decomposition of effective divisors, i.e.\, if $D'$ is an effective divisor on $S$, with (divisorial) Zariski decomposition $P(D')+N(D')$, $\iota(P(D'))+\iota(N(D'))$ is the divisorial Zariski decomposition of $\iota(D')$. Indeed, it suffices to check that the class of $\iota(P(D'))$ belongs to $\overline{\mathrm{Mov}(X)}$. Write $P(D')=\sum_ia_iE_i$, with $a_i>0$ and $E_i$ prime for any $i$. Then $\iota(P(D'))$ is effective, hence we are left to check that $\iota(P(D'))$ $q_X$-intersect non-negatively any prime exceptional divisor on $X$. But this follows from the fact that $P(D')$ is nef on $S$ and \cite[Proposition 4.2, item (ii)]{Bouck}. It follows that a divisor $D'$ on $S$ is big if and only if $\iota(D')$ is because $\iota$ preserves the divisorial Zariski decomposition, and by \cite[Proposition 3.8]{Huy1}. Then $\mu_C(D)=\mu_{\iota(C)}(\iota(D))$, and since $\iota'$ respects the BBF forms of $S$ and $X$, we obtain the equality $\Delta_{F_{\bullet}}(D)=\Delta_{E_C}^{\mathrm{num}}(\iota(D))$. In particular, the natural injective linear map $N^1(S)_{\mathbf{R}} \times \mathbf{R}^2 \to N^1(X)_{\mathbf{R}} \times \mathbf{R}^2$ maps the cone $\Delta_{F_{\bullet}}(S)=\Delta_C^{\mathrm{num}}(S)$ into $\Delta_{\iota(C)}^{\mathrm{num}}(X)$.

\end{ex}
 
 The next example shows that the polygons $\Delta_{E}^{\mathrm{num}}(D)$ might be irrational. 
 
 \begin{ex}
 Consider the Fano variety of lines $F=F(Y)$ of a smooth cubic fourfold $Y$ containing a smooth cubic scroll $T$ with
$$H^4(Y,\mathbf{Z}) \cap H^{2,2}(Y,\mathbf{C})=\mathbf{Z} h^2 + \mathbf{Z} [T].$$ Here $h^2$ is the square of the hyperplane class $h$ of the cubic fourfold. By the work of Beauville and Donagi, $F$ is a projective IHS manifold of dimension $4$. We have $\mathrm{Big}(F)=\mathrm{Eff}(F)=\mathrm{Mov}(F)$ (see \cite{HT2010} and \cite[Theorem 1.2]{Den2}), and $\overline{\mathrm{Eff}(F)}$ is irrational. Then, there exists a big integral divisor $D$ on $X$, and a prime divisor $E$ such that $\Delta_{E}^{\mathrm{num}}(D)$ is an irrational trapezium. Indeed, let $\tau$ be $\alpha([T])$,  where $\alpha \colon H^4(Y,\mathbf{Z})\to H^2(F,\mathbf{Z})$ is the Abel-Jacobi map. Consider also $g=\alpha(h^2)$. Note that $q_F(g)=q_F(g,\tau)=6$, and $q_F(\tau)=2$ (see \cite{HT2010}). The pseudo-effective cone of $F$ is $\overline{\mathrm{Eff}(F)}=\mathrm{cone}\left(g-\left(3-\sqrt6\right)\tau,\left(3+\sqrt6\right)\tau-g\right)$. Let $k$ be a positive integer such that $k\tau$ is represented by a prime divisor, and consider $\Delta_{k\tau}^{\mathrm{num}}(kg)$. Then $\mu_{k\tau}(kg)=3-\sqrt{6}$ is an irrational number. Observe that $q_X(kg-\mu_{k\tau}(g)k\tau,k\tau)=k^2 2 \sqrt{6}>0$, hence $\Delta_{k\tau}^{\mathrm{num}}(kg)$ is a trapezium. To show that $\Delta_{k\tau}^{\mathrm{num}}(kg)$ is irrational it suffices to find an irrational vertex. But this is easy: just take $(\mu_{k\tau}(kg),0)=(3-\sqrt{6},0) \in \Delta_{k\tau}^{\mathrm{num}}(kg)$.
 \end{ex}
 Assuming in the example above that $k=1$, we provide below a picture of the polygon $\Delta_{\tau}^{\mathrm{num}}(g)$.
 
  \clearpage

 \begin{figure}
\begin{tikzpicture}[scale=0.80]
    \draw[gray!50, thin, step=1] (0,0) grid (3,6.5);
    \draw[very thick,->] (0,0) -- (3.4,0) node[right] {$t$};
    \draw[very thick,->] (0,0) -- (0,6.4) node[above] {$y$};
    \foreach \x in {1} \draw (\x,0.05) -- (\x,-0.05) node[below] {\tiny\x};
    \foreach \y in {1,...,6.4} \draw (-0.05,\y) -- (0.05,\y) node[left] {\tiny\y};

    \fill[violet,opacity=0.5] (0,6) -- (2.4494, 1.1012) -- (2.4494,0) -- (0,0);

    \draw (0,6) -- (2.4494, 1.1012);
    \draw (2.4494, 1.1012) -- (2.4494,0);
    
    \node at ((2.4494,-0.32) {$\scriptscriptstyle\mu_{\tau}(g)$};
    \node at (1.2,1.5) {$\scriptstyle\Delta_{\tau}^{\mathrm{num}}(g)$};
\end{tikzpicture}
   \caption{The polygon $\scriptstyle\Delta_{\tau}^{\mathrm{num}}(g)$.}
\end{figure}
 
\printbibliography

\end{document}